\documentclass[leqno]{article}

\usepackage{amsmath,amssymb,amsthm,bm,mathrsfs}
\newcommand{\R}{\mathbb{R}}
\newcommand{\N}{\mathbb{N}}

\newcommand{\PP}{\mathbb{P}}

\newcommand{\Div}{\mathrm{div}\,}

\newcommand{\pt}{\partial}

\DeclareMathOperator*{\esup}{\text{ess\,sup}}

\newtheorem{proposition}{Proposition}[section]
\newtheorem{theorem}{Theorem}[section]
\newtheorem{lemma}{Lemma}[section]
\newtheorem{corollary}[theorem]{Corollary}

\theoremstyle{definition}
\newtheorem{remark}{Remark}[section]

\makeatletter

\@addtoreset{equation}{section}
\makeatother

\usepackage{layout}
\usepackage{xcolor}
\title{Annihilation of slowly-decaying terms of Navier-Stokes flows \\ by external forcing}
\author{
Lorenzo Brandolese \\
{\normalsize Institut Camille Jordan} \\
{\normalsize Universit\'{e} Lyon 1 }\\
{\normalsize E-mail:~\texttt{brandolese@math.univ-lyon1.fr}}
\and
Takahiro Okabe \\
{\normalsize Graduate School of Engineering Science}\\
{\normalsize Osaka University}\protect\\
{\normalsize E-mail:~\texttt{okabe@sigmath.es.osaka-u.ac.jp}}
}
\date{}
\begin{document}
\maketitle
\begin{abstract}
The goal of this paper is to provide an algorithm that, for any sufficiently localised, divergence-free small initial data, explicitly constructs a localised external force
leading to a rapidly dissipative solutions of the Navier--Stokes equations $\R^n$: namely, the energy decay rate of the flow will be forced to satisfy $\|u(t)\|_2^2 =o(t^{-(n+2)/2})$ as $t \to \infty$, which is beyond the usual optimal rate.
 An important feature of our construction is that this force can always be taken compactly supported in space-time, and its profile arbitrarily prescribed up to a spatial rescaling.
Since the forcing term vanishes after a finite time interval, our result suggests that nontrivial interactions between the linear and nonlinear parts occur, annihilating all the slowly decaying terms
contained in  Miyakawa and Schonbek's asymptotic profiles.

\end{abstract}
\section{Introduction}
Let $n\geq 2$. We consider the incompressible Navier-Stokes equations in $\R^n$:
\begin{equation*}\tag{N-S}
\left\{
\begin{split}
&\pt_t u -\Delta u + u\cdot \nabla u 
+ \nabla \pi=\nabla \cdot {f}
\qquad \text{in } \R^n\times (0,\infty),\\
& \Div u=0
\qquad \text{in } \R^n\times (0,\infty),\\
&u(\cdot,0)=a 
\qquad \text{in } \R^n,
\end{split}\right.
\end{equation*}
where $u=u(x,t)=\bigl(u_1(x,t),\dots,u_n(x,t)\bigr)$ and $\pi=\pi(x,t)$ 
denote the unknown velocity and the pressure of the fluid at $(x,t)\in \R^n\times (0,\infty)$, respectively, while, 
$f=f(x,t)=\bigl( f_{k\ell}(x,t)\bigr)_{k,\ell =1,\dots,n}$ 
denotes the external forcing tensor 
and $a=a(x)=\bigl(a_1(x),\dots,a_n(x)\bigr)$ denotes the given initial data.

Starting with the celebrated work 
of Leray \cite{Leray 1934}, the time decay problem has been a major issue in the mathematical study of fluid flows.
Masuda \cite{Masuda}, Schonbek \cite{Schonbek 1985}, Kajikiya and Miyakawa \cite{Kajikiya Miyakawa} and Wiegner \cite{Wiegner}, for instance, obtained pioneering contributions in this direction.
Their results imply that, in the absence of external forcing, the optimal decay rate for a weak solution is 
\begin{equation}\label{eq;optimaldecay}
\|u(t)\|_2 \leq C(1+t)^{-\frac{n+2}{4}}, \qquad t>0,
\end{equation}
for initial data~$a$ in $L_\sigma^2(\R^n)$ under suitable additional conditions. To this purpose, the moment condition $\int_{\R^n} (1+|x|)|a(x)|\text{\,d}x<\infty$, for example, would be enough.
Subsequently, Fujigaki-Miyakawa \cite{Fujigaki Miyakawa SIAM} clarified that the decay rate in the right-hand side of~\eqref{eq;optimaldecay}
actually describes the decay rate of the nonlinear terms. Indeed, they derived the asymptotic expansion of the linear part and of the nonlinear part as follows
\begin{equation*}
\lim_{t\to\infty}
t^{\frac{1}{2}+\frac{n}{2}(1-\frac{1}{q})}
\left\|
u_j(t)+\sum_{k=1}^n(\pt_k E_t)(\cdot)\int_{\R^n}y_k a_j(y)\,dy
+
\sum_{\ell,k=1}^n F_{\ell k,j}(\cdot,t)\int_0^\infty\!\!\!\int_{\R^n}
(u_\ell u_k)(y,s)\text{\,d}y\text{\,d}s
\right\|_q=0,
\end{equation*}
for all $j=1,\dots, n$ and for all $1 \leq q \leq \infty$,
where $E_t(x)=(4\pi t)^{-n/2}\exp \bigl(-\frac{|x|^2}{4t}\bigr)$ is the heat kernel and 
\[
F_{\ell k, j}(x,t)=
\pt_\ell E_t(x)\delta_{j k} + \int_t^\infty
\pt_\ell \pt_k \pt_j E_s(x)\text{\,d}s,
\]
where $\delta_{jk}$ is the Kronecker symbol.

Miyakawa and Schonbek \cite{Miyakawa Schonbek}
deduced from the above asymptotic profile necessary and sufficient conditions ensuring that
the flow is \emph{rapidly dissipative}, in the sense that
\[
\lim_{t\to\infty}
t^{\frac{1}{2}+\frac{n}{2}(1-\frac{1}{q})}
\left\|
u(t)
\right\|_q=0, \qquad 1\le q\le\infty.
\]
For initial data with finite moments up to the first order, these conditions read
\begin{subequations}
\begin{equation}\label{MSconditionL}
\int_{\R^n} y_k a_j(y)\,dy =0 \qquad \text{for } j,k=1,\dots, n, 
\end{equation}
and
\begin{equation}\label{MSconditionN}
\exists\,c\in\R\quad\text{such that}\quad
\int_0^\infty\!\!\!\int_{\R^n} (u_k u_\ell)(y,s)\text{\,d}y\text{\,d}s = c\,\delta_{k \ell}
\qquad \text{for } k,\ell=1,\dots, n.
\end{equation}
\end{subequations}

Condition~\eqref{MSconditionN} is difficult to check, as it requires information of the (unknown) flow 
over whole space-time region.
For this reason, it is usually not possible to predict whether or not a given flow is rapidly dissipative.
To overcome this difficulty, and to make evidence that condition~\eqref{MSconditionN} is not reduced to the identically zero solution,
one could restrict the problem to flows invariant
under the action of suitable symmetry groups.
For example, in 2D one could consider flows with radial vorticity, but this idea is no longer effective in 3D because of topological obstructions.  The first author, to circumvent this difficulty, considered to this purpose cyclically symmetry of the flow, i.e.,
\begin{itemize}
\item[(a)] $u_j(\cdot,t)$ is odd in $x_j$ and even in each other variables,
\item[(b)] 
$
u_1(x_1,\dots,x_n,t)
=u_2(x_n,x_1,\dots,x_{n-1},t)
=\dots
=u_n(x_2,\dots,x_n,x_1,t).
$
\end{itemize}
These conditions are easier to check, as they are preserved during the evolution (for the unique strong solution, if this is known to exist, but also for any Leray solutions constructed so far)
if they are satisfied for the initial data.
In the class of flows satisfying (a) and (b), 
the first author \cite{Brandolese 2001}, Miyakawa \cite{Miyakawa 2002 FE, Miyakawa 2002 Hiroshima} made evidence of the existence of rapidly dissipative solutions. They also observed that the faster decay of these solutions agrees with the rate of the
second-order, or of the third order terms, in the asymptotic expansion of the flow. 
Later, the second author and Tsutsui \cite{Okabe Tsutsui JDE} gave a generalization of 
\cite{Brandolese 2004}, \cite{Miyakawa 2002 FE, Miyakawa 2002 Hiroshima} with weighted Hardy spaces.
More general group actions were discussed by the first author \cite{Brandolese MA 2004}.

However, symmetry conditions, like (a) and (b), look somewhat artificial.
In particular, these symmetric flows are non-stable in the class of generic flows: to make them physically realistic, an additional control $f(x,t)$ acting in the whole time interval $(0,+\infty)$ should be required:
it is therefore a natural problem to see how a generic flow (featuring no
special symmetry) could evolve into a rapidly dissipative flow  during the evolution, thanks to some other process, implying the annihilation of the slowly decaying terms of its asymptotic profiles.
The purpose of this paper is to address this issue.

As mentioned above, the essential difficulty will be 
the verification of \eqref{MSconditionN}.
Instead of the cyclic symmetry, 
for any initial velocity which is small in a suitable sense, we construct an associated external force and a rapidly dissipative solution of the forced Navier-Stokes equations.

This approach not only looks mathematically natural, 
but also realistic in physics or engineering, as
the flow will be forced to slow down in the large time at faster rates,   
through the introduction of a \emph{ad hoc} forcing term, depending on the given initial state, that will act only on a bounded region and over a finite time interval, that can be taken arbitrarily short.

For a given initial velocity field $a$, our strategy will be to provide an algorithm, leading in the limit to the construction of a force of divergence form  $\nabla \cdot f$, compactly supported in space-time, so that, for $j=1,\ldots,n$,
\begin{equation}\label{eq;fasympt}
\int_0^t  [e^{(t-s)\Delta}\PP \nabla \cdot f]_j(s)\text{\,d}s 
\sim 
\sum_{k,\ell=1}^n F_{\ell k, j}(\cdot,t)
\int_0^\infty \int_{\R^n} (u_\ell u_k)(y,s)\text{\,d}y\text{\,d}s
\qquad \text{for large }t>0.
\end{equation}
This force will counterbalance the slowly decaying nonlinear terms appearing in Fujigaki and Miyakawa's asympotic profile.
Here $\PP$ is  the Leray-Hopf (also called after Weyl-Helmholtz or Fujita-Kato) projection onto solenoidal vectors.
For the realization of \eqref{eq;fasympt}, 
we introduce the following computable procedure:
\begin{equation}\label{OURmethod}
u^{(m)}(t)=e^{t\Delta}a 
+
\int_0^t  e^{(t-s)\Delta}\PP
\nabla \cdot f^{(m)}(s)\text{\,d}s
-\int_0^t  e^{(t-s)\Delta} \PP \nabla \cdot
(u^{(m)}\otimes u^{(m)})(s)\text{\,d}s \qquad m=1,2,\dots.
\end{equation}
Here, the forcing tensor $f^{(m)}=(f^{(m)}_{k\ell})$ is given by $f^{(0)}\equiv 0$ and
\begin{equation}
\label{algo}
f_{k\ell}^{(m)}(x,t)=\begin{cases}
c_{k\ell}^{(m-1)}\phi(x,t), & k\neq \ell,\\
(c_{kk}^{(m-1)}-\bar{c}^{(m-1)})\phi(x,t), & k=\ell,
\end{cases}
\end{equation} 
for some function $\phi\in C_0^\infty(\R^n\times [0,\infty))$, where 
$c_{k \ell}^{(m)}
=\int_0^\infty\!\!\int_{\R^n} 
(u_k^{(m)} u_\ell^{(m)})(y,s)\text{\,d}y\text{\,d}s$ 
and
$\bar{c}^{(m)}=c_{11}^{(m)}+\dots+c_{nn}^{(m)}$.
We note that since we are able to take $\phi$ compactly supported in both space and time, 
in order to control the flow,  
it is enough that the force is applied to finite time and bounded space region. 
Remarkably, as we will see, the profile $\phi$ 
can be prescribed in an essentially arbitrary way.

For our approach to be effective,
 it will be crucial to derive a bound $K$, \emph{independent of $m$},  such that
\begin{equation}\label{BoundC}
|c_{k\ell}^{(m)}| =\biggl|\int_0^\infty \!\!\!\int_{\R^n}
(u_k^{(m)} u_\ell^{(m)})(y, s)\text{\,d}y\text{\,d}s \biggr|
\leq
\int_0^\infty \| u^{(m)}(s) \|_2^2 \text{\,d}s
\le K.
\end{equation}
Condition~\eqref{BoundC} would follow from suitable $L^2$-decay estimates. 
A natural idea would consist in relying on a Fourier splitting technique to get estimates
for the $L^2$-norm of $u^{(m)}$, with the needed rate. See
\cite{Schonbek 1985,Kajikiya Miyakawa,Wiegner}. 
But in our context, the Fourier splitting method seems to have a slight drawback: 
in decay estimates like  \eqref{eq;optimaldecay}, the constant $C$ appearing in the right-hand side depends not only on the size of the data, but also, in a quite complicated way, on the shape of the data and of the force. This is a serious issue, as in our case we have to deal with a recursively defined  sequence of forces: all the constants in $L^2$ decay estimates a priori depend on~$m$ and this
makes estimate~\eqref{BoundC} not so
straightforward.

On the other hand, one could also apply a  
Fujita-Kato method, see Fujita and Kato \cite{Fujita Kato} and
Kato \cite{Kato}. In this case the difficulty is that the size of $f^{(m)}$, i.e.,  of the coefficients $c_{k\ell}^{(m-1)}$ has \emph{a priori} an influence on the lifetime of the mild solutions $u^{(m)}$, so an important technical issue arises: to prove that
the recursively defined solutions  $u^{(m)}$ are indeed global in time, for all natural integer $m$.

Due to the above difficulties, 
we need to develop an alternative approach
and to establish 
time decay estimates, as in
\eqref{eq;optimaldecay}, by carefully making explicit how 
the constants depend on the given data $a$, on
$\phi$ and on the dimension $n$.
For this purpose, weighted Hardy spaces 
would be an effective tool.
Indeed, the second author and Tsutsui \cite{Okabe Tsutsui JDE} introduced weighted Hardy spaces to derive higher order asymptotic expansions:
weighted Hardy spaces
enable us to deal with higher order weights 
and to obtain more rapid decay 
compared with the weighted Lebesgue spaces.
With the aid of the weighted Hardy norm,
we could make a specific refinement 
of the Fujita-Kato iteration scheme, 
first giving bounds as in \eqref{BoundC},
next ensuring the convergence 
of our procedure \eqref{OURmethod} toward an external force $f$ and to a
rapidly dissipative solution~$u$.

In this paper, however, we will adopt an alternative strategy: this will consist in the systematic derivation
of \emph{quantified scale-invariant decay estimates} for the $L^2$-norm of the solution.
Let us illustrate what we mean by a specific example:
in the case $n=3$, for well localised initial data $a$ such that $\|a\|_3<\!\!\!<1$ and in the case $f\equiv0$,
we will derive the estimate
\begin{equation}
\label{3scale}
 \|u(t)\|_2\le \min\Bigl\{\|a\|_2\,,\, c K(a)t^{-5/4}\Bigr\}
\end{equation}
where $c$ is an absolute constant and
\[
K(a)=\textstyle\int_{\R^3}|x|\,|a(x)|{\rm\,d}x+
\|a\|_1^{1/2}\|a\|_2\Bigl(\textstyle\int_{\R^3}|x|\,|a(x)|{\rm \,d}x\Bigr)^{1/2}
+\|a\|_{4/3}^{2}\|a\|_2.
\]
Estimate~\eqref{3scale} is what we call the ``quantified scale-invariant'' version of classical Wiegner's estimate, 
$\|u(t)\|_2\le C(1+t)^{-5/4}$. Its crucial advantage is that the dependence of the constant $C$ on the initial data is made completely explicit. Notice that the three terms defining $K(a)$ rescale
in the same way, so~\eqref{3scale} is indeed scale-invariant. 
Uniform bounds of the form~\eqref{BoundC} will rely on such type of estimates.

This paper is organised as follows.
We state our main result, Theorem~\ref{thm;2}, in Section 2.
In Section 3 we recall some classical estimates.
Section 4 is entirely devoted to the proof of our main result.
We begin in Subsection~\ref{sub:0} with considering the case  of a vanishing force $f^{(0)}=0$, and we present a slight refinement of Fujita-Kato scheme, with the purpose of 
putting in evidence the dependance of all constants
 appearing in the fixed point argument. Here we also establish a class of quantified scale invariant estimates, in the same spirit as~\eqref{3scale}.
In Subsection~\ref{sec:induci}, we introduce the sequence of external tensors 
$\{f^{(m)}\}$ and we construct the associated solutions
$u^{(m)}(t)$ of (N-S), $m=0,1,2,\dots$, of \eqref{OURmethod}.
The main step of this section is the conclusion that these solutions $u^{(m)}$ are all global in time,
under appropriate smallness conditions on the data and that all the relevant estimates are independent on~$m$.
Finally, we complete the proof of Theorem \ref{thm;2} by proving the convergence of $u^{(m)}$ and
of $f^{(m)}$ and the decay estimate for the limit solution.

\section{The main result}

Our starting point will be the following result of Fujigaki and Miyakawa~\cite{Fujigaki Miyakawa SIAM}.

\begin{theorem}
\label{the;FMF}
Let $a \in L^1(\R^n)\cap L^n_{\sigma}(\R^n)$ with
$\int_{\R^n} |x||a(x)|\,dx<\infty$
and $f \in C_c^\infty
\bigl(\R^n\times [0,\infty)\bigr)$.
Suppose $u\in BC\bigl([0,\infty);L^n_\sigma(\R^n)\bigr)$ is a global mild solution of (N-S). If $n\geq 5$, assume also that $\liminf\limits_{t\to\infty}\|u(t)\|_n$ is sufficiently small.
Then it holds that
\begin{multline*}
\lim\limits_{t\to\infty}
t^{\frac{1}{2}+\frac{n}{2}(1-\frac{1}{q})}
\Big\|u_j(t)
+\sum_{k=1}^{n}(\pt_k E_t)(\cdot)\int_{\R^n}y_ka_j(y)\,dy
\\
-
\sum_{k,\ell=1}^nF_{\ell k,j}(\cdot,t)\int_0^\infty\!\!\!\int_{\R^n} f_{k\ell}(y,s)\text{\,d}y\text{\,d}s
+
\sum_{k,\ell=1}^n F_{\ell k,j}(\cdot,t)\int_0^\infty \!\!\!\int_{\R^n} (u_\ell u_k)(y,s)
{\,\rm d}y{\,\rm d}s\Big\|_q=0
\end{multline*}
for $1 \leq q \leq \infty $.
\end{theorem}

 Though \cite{Fujigaki Miyakawa SIAM} dealt with only the case $f\equiv 0$, the proof is essentially same. 
The derivation of the leading order term for the Duhamel term of $f$ is just analogue to that for the nonlinear term. We omit the proof.
As an immediate consequence from Miyakawa and Schonbek \cite{Miyakawa Schonbek}, the condition associated with
\eqref{MSconditionN} is modified as follows:
\begin{corollary}\label{cor;MS}
Let $a$, $f$ and $u$ be as in Theorem~\ref{the;FMF}. It holds that,
\begin{equation*}
\label{q-adec}
\lim_{t\to \infty} 
t^{\frac{1}{2}+\frac{n}{2}(1-\frac{1}{q})}
\| u(t)-e^{t\Delta}a\|_q=0 , \qquad (1\leq q \leq \infty)
\end{equation*}
if and only if there exists $c\in\R$ such that
\begin{equation}\label{eq;MSconditionF}
\int_0^\infty \!\!\!\int_{\R^n} 
f_{k\ell}(y,s){\,\rm d}y{\,\rm d}s
-\int_0^\infty\!\!\!\int_{\R^n} (u_\ell u_k)(y,s){\,\rm d}y{\,\rm d}s=c\,\delta_{k \ell}
\qquad(k,\ell=1,\dots,n).
\end{equation}
Moreover,
\begin{equation*}
\label{q-dec}
\lim_{t\to \infty} 
t^{\frac{1}{2}+\frac{n}{2}(1-\frac{1}{q})}\| u(t)\|_q=0  \qquad (1\leq q \leq \infty)
\end{equation*}
if and only if condition~\eqref{eq;MSconditionF} holds and also
$\int_{\R^n} y_k a(y){\rm \,d}y =0$  for all $k=1,\ldots,n$.

Furthermore, if \eqref{eq;MSconditionF} does not hold, or if at least one of the first-order moments of $a$ does not vanish, then 
$\liminf\limits_{t\to\infty}
t^{\frac{1}{2}+\frac{n}{2}(1-\frac{1}{q})}
\|u(t)\|_q >0$
for $1\leq q \leq \infty$.
\end{corollary}
\begin{remark}
We note that if 
\begin{equation*}
\int_0^\infty \!\!\!\int_{\R^n}
f_{k \ell}(y,s){\rm \,d}y{\rm \,d}s 
\neq
\int_0^\infty\!\!\! \int_{\R^n} 
f_{\ell k}(y,s){\rm \,d}y{\rm \,d}s \qquad\text{for some $k$ and $\ell$},
\end{equation*} 
then the condition \eqref{eq;MSconditionF} does not hold. Hence,
if the tensor $f$ is not symmetric in the above sense, one cannot expect
rapid time decay, no matter how fast $f$ decays at spatial infinity and time infinity. 
In this case, by the last assertion of the corollary, we have
\begin{equation*}
\liminf\limits_{t\to\infty}
t^{\frac{1}{2}+\frac{n}{2}(1-\frac{1}{q})}
\|u(t)\|_q >0, \qquad 1\leq q \leq \infty.
\end{equation*}

\end{remark}
Corollary \ref{cor;MS} is  the natural generalisation 
of \cite{Miyakawa Schonbek}, when an external force acts on the flow. It will be our tool to show how to control the large time decay of the flow (obtaining rapidly dissipative solutions) by forcing the fluid only during a short time interval.
We now state our main result.

\begin{theorem}
\mbox{}
\label{thm;2}
\begin{itemize}
\item[(i)]
Let $n\ge2$.
There exists $\delta=\delta(n)>0$ with the following property.
If $a \in L^n_{\sigma}(\R^n)$ satisfies 
$\int_{\R^n} (1+|x|)|a(x)|{\rm \,d}x <\infty$
 and the smallness condition
\begin{equation}
\label{small-a}
\|a\|_{n}
\leq \delta,
\tag{S}
\end{equation} 
then there exists a forcing term of divergence form 
$\nabla\cdot f$, with  compact support in space-time, $f\in C^\infty_c\bigl(\R^n\times [0,\infty)\bigr)$, such that
the unique global solution of the Navier--Stokes equations (N-S) satisfies
\begin{equation*}
\label{fast-decay-a}
\lim_{t\to\infty}
t^{\frac{1}{2}+\frac{n}{2}(1-\frac{1}{q})}
\left\|
u(t)-e^{t\Delta}a
\right\|_q=0, \qquad 1\le q\le\infty.
\end{equation*}
In particular, if the $n$ first-order moments of $a$ all vanish, then
$u$ is rapidly dissipative:
\begin{equation}
\label{fast-decay}
\lim_{t\to\infty}
t^{\frac{1}{2}+\frac{n}{2}(1-\frac{1}{q})}
\left\|
u(t)
\right\|_q=0, \qquad 1\le q\le\infty.
\end{equation}

\item[(ii)]
Moreover, the shape of the above forcing term
 can be \emph{arbitrarily prescribed}, in the following sense:
for any compactly supported scalar function $\Psi\in L^\infty_c(\R^n\times \R^+)$, such that $\int_0^\infty\!\!\!\int_{\R^n}\Psi(y,s){\rm \,d}y{\rm \,d}s\not=0$,
there exist  $R>0$ and coefficients $\lambda_{\ell k}\in\R$  such that, in Item~(i), the forcing tensor $f=(f_{\ell k})$ can be taken of the form
\[
f_{\ell k}=\lambda_{\ell k}\,\Psi(\cdot/R,t), \qquad (\ell,k=1,\ldots,n).
\]
\end{itemize}
 \end{theorem}
\begin{remark}
From the second assertion, it follows that one can force the flow to have a fast  decay in large time,
by acting with an external force in a bounded region, during a time interval that can be taken arbitrarily short.
\end{remark}

Let us discuss the case of an initial data with first order vanishing moments, and let us apply
a forcing term to the flow, as described by our theorem. We thus get a rapidly dissipative flow. Moreover,
in Theorem~\ref{thm;2}, since~$f$ identically vanishes after a finite time interval, 
the solution eventually behaves like a non-forced Navier--Stokes flow.
More precisely, after some time $t>t_0$, 
the flow is governed by
\begin{equation*}
u(t)=e^{(t-t_0)\Delta} u(t_0) -
\int_{t_0}^t e^{(t-s)\Delta}
\PP\nabla\cdot(u\otimes u)(s)\text{\,d}s, \qquad t>t_0.  
\end{equation*}
This does not mean that the effect of the force disappears for $t>t_0$, as the new initial data $u(t_0)$ does depend on the past action of $f$ during the interval $[0,t_0]$.

So, due to the result of \cite{Miyakawa Schonbek}, we have  two possibilities.
\begin{itemize}
\item[i)]
The first possible scenario is that the moment condition for $u(t)$ breaks down
at some times $t_1\ge t_0$, i.e.,
$\int_{\R^n} (1+|x|)|u(x,t_1)|\text{\,d}x=\infty$.
At those time instants~$t_1$, 
the linear part $t\mapsto e^{t\Delta}u(t_1)$ and the nonlinear part
$t\mapsto \int_{t_1}^t  e^{(t-s)\Delta}
\PP\nabla\cdot (u\otimes u)\text{\,d}s$, individually, may both decay slowly, but they 
feature a nontrivial interaction 
annihilating the slowly decay terms, thus leading to a fast decay of solution of the free Navier--Stokes equations starting from $u(t_1)$.
\item[ii)]
The second possible scenario is that the finiteness of first order moment
of $u(t)$ is preserved, i.e., $\int_{\R^n} (1+|x|)|u(x,t)|{\rm \,d}x<\infty$ \emph{for all} $t\geq t_0$.
In this case, a much stronger condition than~\eqref{MSconditionN} must be true,
namely:
\begin{equation*}
\int_{\R^n} (u_\ell u_k)(x,t){\rm \,d}x =0 \quad 
\text{for } k\neq \ell,
\quad\text{and}\quad 
\int_{\R^n} u_1(x,t)^2
{\rm \,d}x=\ldots
=\int_{\R^n}u_n(x,t)^2{\rm \,d}x,
\quad \text{for all $t\geq t_0$}.
\end{equation*}
(See \cite{Brandolese 2001}).
This second scenario looks non-generic as it contrasts with the general spatial spreading phenomenon of the velocity field: the only known examples of flows satisfying the above orthogonality relations 
 are those
constructed putting symmetries as in~\cite{Brandolese 2001, Brandolese 2004,Miyakawa 2002 Hiroshima}.

\end{itemize}

%
%
%
%
%
%

Let us stress the fact that, if one allows non-compactly supported external forces (not necessarily of divergence form), then
it would be a trivial task to achieve the goal of forcing the flow to be rapidly dissipative, as in~\eqref{fast-decay} (or even to bring it to rest in finite time). Indeed one could just first define any divergence-free vector field $v(x,t)$ which is equal to $a(x)$ as $t=0$ and equal to $0$ for $t\ge t_0$, next define the force
to be the residual of the Navier--Stokes operator.
But of course,  such a force  would be spread out in the whole space because of the nonlocal nature of the pressure.

{

\section{Notations and preliminary estimates} 

Let us introduce some notations and function spaces.
Let $C_c^\infty(\Omega)$ denote the set of 
all $C^\infty$-functions (or vectors) with compact support in a connected set $\Omega$.
Let $C_{c,\sigma}^\infty(\mathbb{R}^n)$ denote 
the set of all $C^\infty$-solenoidal vectors $\varphi$ with compact support
in $\mathbb{R}^n$, i.e., $\Div \varphi=0$ in $\mathbb{R}^n$.
The space $L^r_{\sigma}(\mathbb{R}^n)$ is the closure of 
$C^\infty_{c,\sigma}(\mathbb{R}^n)$ with respect to the $L^r$-norm 
$\|\cdot \|_r$, $1< r < \infty$;
$L^r(\mathbb{R}^n)$  denote 
the usual (vector-valued) Lebesgue space  over 
$\mathbb{R}^n$. 
Moreover, $C(I;X)$, $BC(I;X)$ and $L^r(I;X)$ denote 
the $X$-valued continuous and bounded continuous functions over the interval $I\subset \R$, 
and $X$-valued $L^r$-functions, respectively.

In the estimates of this paper, we will mainly compare functions of the time variable.
When we write
\[
A(t) \lesssim B(t)
\]
we mean that there exists a constant $c>0$, \emph{only depending on the space dimension}~$n$, such that $A(t)\le c B(t)$ for all $t$. 
In particular, when the functions $A$ and $B$ depend on other parameters (such as the initial data $a$,
the recursive parameter $m$, etc.), the constant~$c$ will be independent on these parameters.

When the constants in our estimates depend on parameters other than the space dimension
(as in~\eqref{eq;lplq} below), we will indicate this fact explicitly in our notations.

We start recalling some well-known $L^p$-$L^q$ estimates, 
that play an important role through this paper.
\begin{proposition}\label{prop;LpLq}
Let $1\leq p \leq q \leq \infty$. 
Then there exists a constant $C_{q,p}>0$ such that
\begin{alignat}{2} \label{eq;lplq}
\|e^{t\Delta}a\|_q 
&\leq C_{q,p}t^{-\frac{n}{2}(\frac{1}{p}-\frac{1}{q})} \|a\|_p, 
\qquad t>0,\\
\label{eq;nlplq}
\|\nabla e^{t\Delta} a\|_q 
&\leq C_{q,p} t^{-\frac{n}{2}(\frac{1}{p}-\frac{1}{q})-\frac{1}{2}}\|a\|_p,
\qquad t>0
\end{alignat}
for a function, velocity vector or tensor $a\in L^p(\R^n)$.
\end{proposition}
The proof of \eqref{eq;lplq} and \eqref{eq;nlplq} 
are immediately derived from the Young inequality 
for the heat kernel and $a$. For the Stokes semigroup on $L^p_\sigma(\R^n)$, we exclude the case $(p,q)=(\infty,\infty)$ and $(p,q)=(1,1)$. 
See, for instance, \cite{Kato}, \cite[Corollary 1.1]{Tsutsui JFA}.

We also recall a well known variant of the previous estimates, that we will use often for the $L^2$-norm:
\begin{lemma}
\label{lem:heat2}
If $a\in L^1(\R^n)$ with $\int_{\R^n}|x|\,|a(x)|{\rm \,d}x<\infty$ and $\Div a=0$,
then
\[
\|e^{t\Delta}a\|_2\lesssim\textstyle\Bigl(\int_{\R^n}|x||a(x)|{\rm\,d}x\Bigr)t^{-(n+2)/4}.
\] 
\end{lemma}
For the proof, see \cite{DuoZ92}. We give a proof below for reader's convenience:
\begin{proof}
Since $a\in L^1(\R^n)$ and $\Div a=0$, we note that $\int_{\R^n} a(y){\rm\,d}y=0$. 
Hence, we easily obtain
\begin{equation*}
\begin{split}
e^{t\Delta} a (x)
&=\int_{\R^n}E_t(x-y)a(y){\rm \, d}y
=\int_{\R^n} \bigl(E_t(x-y)-E_t(x)\bigr)a(y) {\rm \,d}y
\\
&=-\int_{\R^n}\!\!\int_0^1 \nabla E_t (x-\theta y)
\cdot ya(y){\rm \, d}\theta{\rm\, d}y.
\end{split}
\end{equation*}
Then the Minkovski inequality (for integrals) implies that
\begin{equation*}
\|e^{t\Delta}a\|_2 \leq
\int_{\R^n}\int_0^1 |y|\,|a(y)|
\left[\int_{\R^n} |\nabla E_t(x-\theta y)|^2{\rm\,d}x\right]^{1/2}{\rm \, d}\theta{\rm \,d}y
=\|\nabla E_t\|_2\int_{\R^n}|y| \,|a(y)| {\rm \,d}y.
\end{equation*} 
Since $\|\nabla E_t\|_2 = t^{-\frac{n+2}{4}}\|\nabla E_1\|_2$, the proof is completed.
\end{proof}

For $1<r<\infty$, the Leray projection $\PP:L^r(\R^n)
\to L^r_{\sigma}(\R^n)$ satisfies 
$\|\PP u\|_r\le A_r\|u\|_r$ for all $u \in L^r(\R^n)$
with some constant $A_r>0$.

We consider the integral formulation of (N-S), that can be written
in the abstract form
\begin{equation}
 \label{NS-ab}
 u(t)=e^{t\Delta}a+\int_0^te^{(t-s)\Delta}\PP\nabla\cdot f(s){\rm\,d}s+G(u,u)(s),
 \qquad\nabla\cdot a=0,
\end{equation}
where
\begin{equation}
\label{Gop}
 G(u,v)
=-\int_0^t e^{(t-s)\Delta}\PP \nabla\cdot(u\otimes v)(s)\text{\,d}s.
\end{equation}


We will systematically glue together the heat kernel, the Leray projector and the divergence
operator, obtaining in this way the convolution operator $e^{t\Delta}\mathbb{P}\text{div}$.
The kernel of this operator is denoted by $F$. Its components are given by
$F_{\ell k, j}(x,t)=
\pt_\ell E_t(x)\delta_{j k} + \int_t^\infty
\pt_\ell \pt_k \pt_j E_s(x)\text{\,d}s$,
for $\ell,k,j=1,\ldots,n$.
Such kernel satisfies $F\in C^\infty(\R^n\times(0,\infty))$ and the scaling relations $F(x,t)=t^{-(n+1)/2}F(x/\sqrt t,1)$. Moreover, for all $t>0$,
\begin{equation}
\label{est:F}
 \|F(\cdot,t)\|_p= c_p\,t^{-(n+1)/2+n/(2p)}, \qquad1\le p\le\infty,
\end{equation}
for some constant $c_p>0$ depending only on~$n$ and $p$. See~\cite{Fujigaki Miyakawa SIAM}.

\section{Proof of Theorem \ref{thm;2}}
\label{sec:proof}

\subsection{The case of a vanishing external force}
\label{sub:0}

\paragraph{Construction of the solution of the free Navier--Stokes equations.}

In this paragraph we quickly present a sightly simplified version of Kato's method for the construction of mild solutions in $L^n(\R^n)$.
Let us make use of Kato's space, defined for $n\le p\le\infty$ by,
\[
X_p=\bigl\{v\in L^\infty_{\rm loc}\bigl(\R^+;L^p(\R^n)\bigr)
\colon \|v\|_{X_p}=
\esup\limits_{t>0}\,t^{\frac12-\frac{n}{2p}}\|v(t)\|_p<\infty\bigr\}.
\]
From now on we will abusively write $\sup_{t>0}$ instead of $\text{ess\,sup}_{t>0}$ to simplify our notations.
Notice that $X_n=L^\infty(\R^+;L^n(\R^n))$.
By the standard heat-kernel estimate \eqref{eq;lplq},
\begin{equation}
\label{est:line}
\|e^{t\Delta}a\|_{X_p}\le C_{p,n}\|a\|_{n}\qquad n\le p\le\infty.
\end{equation}
It easily follows from~\eqref{est:F} and the usual H\"older and Young estimates that
\begin{equation}
 \label{kato-est}
 \|G(u,v)\|_{X_r}\le \kappa_{r,p}\|u\|_{X_p}\|v\|_{X_p},\qquad
 \textstyle
 \frac1r\le \frac2p<\frac1n+\frac1r, \quad n<p<\infty,\quad n\le r\le\infty,
\end{equation}
for some constant $\kappa_{r,p}$ depending only on $p,r$ and $n$.
In particular, choosing, e.g., $p=r=2n$ in~\eqref{kato-est},
\begin{equation}
 \label{kato-est-2n}
 \|G(u,v)\|_{X_{2n}}\le \kappa_{2n,2n}\|u\|_{X_{2n}}\|v\|_{X_{2n}}.
\end{equation}
On the other hand,
by the usual heat-kernel estimate
\begin{equation*}
 \|e^{t\Delta}a\|_{X_{2n}}\le C_{2n,n}\|a\|_{n}.
\end{equation*}

In this subsection we just consider (N-S) in the case $f\equiv0$. 
In this case, 
by an appropriate choice of $\delta>0$,
namely, choosing
\begin{equation}
 \tag{D1} 0<\delta<1/(4C_{2n,n}\kappa_{2n,2n}),
\end{equation}}
 from the smallness assumption~\eqref{small-a}
we can ensure that
\begin{equation}
 \label{small1}
 \|a\|_{n}<1/(4C_{2n,n}\kappa_{2n,2n}).
\end{equation}
So the usual fixed point Lemma, applied to the equation 
$u=e^{t\Delta} a+G(u,u)$ in the space $X_{2n}$, implies that
a mild solution $u\in X_{2n}$ to (N-S) (with identically zero external force) does exist.
Moreover,
\begin{equation*}
 \|u\|_{X_{2n}}\le 2C_{2n,n}\|a\|_n \lesssim\|a\|_n
\end{equation*}
and this condition uniquely defines~$u$.
Such solution is obtained as the limit, in the $X_{2n}$-norm, of the sequence of approximate
solutions
\begin{equation}
\label{iterkato}
u_{k+1}=e^{t\Delta}a+G(u_k,u_k),\qquad
k\in\N,
\qquad
\text{with \; }
u_0=e^{t\Delta}a.
\end{equation}
Here are some further estimates on $u$, that directly follow from
the equation
\[
u=e^{t\Delta}a+G(u,u)
\]
and the application of~\eqref{est:line} and \eqref{kato-est} with different choice of the parameters
(we also use~\eqref{small1} for the right inequalities below):
\begin{equation}
\label{argu}
\begin{split}
\|u\|_{X_n} &\le C_{n,n}\|a\|_n+\kappa_{n,2n}\|u\|_{X_{2n}}^2\lesssim\|a\|_n,\\
\|u\|_{X_{3n}} &\le C_{3n,n}\|a\|_n+\kappa_{3n,2n}\|u\|_{X_{2n}}^2 \lesssim\|a\|_n,\\
\|u\|_{X_\infty} &\le C_{\infty,n}\|a\|_n
+\kappa_{\infty,3n}\|u\|_{X_{3n}}^2\lesssim\|a\|_n.
\end{split}
\end{equation}
In particular, as all previous norms are finite,
\[
u\in X_n\cap X_\infty.
\]

\paragraph{First quantified $L^2$-decay estimates.}
The goal of this and next paragraph is to provide some quantified versions of $L^2$-decay rate estimates.  At the end of next paragraph we will be able to quantify Wiegner's fundamental estimate~\eqref{eq;optimaldecay}. More precisely, our goal will be to make explicit how
the constant $C$, in the right-hand side of the optimal decay result 
\begin{equation*}
\|u(t)\|_2\le C(1+t)^{-(n+2)/4},
\end{equation*}
depends on~$a$.
In fact, we will do more than this.
A drawback of the above estimate is that it is not scale-invariant under the usual scaling
$a\mapsto\lambda a(\lambda\cdot)$ and $u\mapsto \lambda u(\lambda\cdot,\lambda^2\cdot)$.
A better way of estimating the $L^2$-norm, respecting the natural scaling of the Navier--Stokes
 equations, would be to look for an estimate of the form
\begin{equation}
 \|u(t)\|_2\le \|a\|_2\wedge K(a)t^{-(n+2)/4}
\end{equation}
(the wedge symbol stands for the minimum), for some functional $K=K(a)$, independent on~$t$, and
 satisfying the scaling relations
\begin{equation}
\label{scaK}
K(a)=\lambda^{n}K\bigl(\lambda\,a(\lambda\cdot)\bigr), \qquad\text{for all $\lambda>0$}.
\end{equation}
We will achieve this at the end of next paragraph, by making explicit the expression of $K(a)$.

We start with estimating the
$L^2$-norm of the approximate solutions $u_k$ defined in~\eqref{iterkato}. 
We have $\|u_0(t)\|_2\le \|a\|_2$. Next,
applying~\eqref{est:F} with $p=1$, and using that $a\in L^1(\R^n)\cap L^n(\R^n)\subset L^2(\R^n)$,
 we get, for $k\in\N$,
\[
\begin{split}
 \|u_{k+1}(t)\|_2
 &\le
\|a\|_{2}
+c_1\int_0^{t}(t-s)^{-1/2}\|u_k(s)\|_2\|u_k(s)\|_\infty{\rm \,d}s\\
&\le 
\|a\|_2+\pi c_1\|u_k\|_{X_\infty}\sup_{s>0}\|u_k(s)\|_2.\\
\end{split}
\]
But the approximate solutions $u_k$ satisfy the same estimates as~\eqref{argu}, uniformly with respect to~$k\in\N$. In particular,
$\|u_k\|_{X_\infty}\lesssim\|a\|_n$. Hence, 
by our smallness condition~\eqref{small-a} and appropriate choice of $\delta>0$\footnote
{This requires to replace our previous condition (D1) on $\delta$ by a more restrictive condition of the form
\begin{equation}
 \tag{D2} 0<\delta<\delta_2,
\end{equation}
where $\delta_2$ is not made explicit for sake of conciseness. In fact, in the subsequent steps of the proof, we will need to further strengthen condition (D2) by more restrictive ones:
 \begin{align}
&\tag{D3} 0<\delta<\delta_3 \qquad(\text{cf.~\eqref{hereD3} below}),\\
&\quad\ldots\nonumber\\
&\tag{D9} 0<\delta<\delta_9  \qquad(\text{cf. the end of the proof of Lemma~\ref{lem-sa}}).
\end{align}
All these constants $\delta_2,\ldots, \delta_9$ just depend on the space dimension.
},
we can ensure that
\[
\pi c_1\|u_k\|_{X_\infty}\le1/2<1
\]
and, by iteration, that 
$\sup\limits_{t>0} \|u_k(t)\|_2
\le 2\|a\|_2$ for all $k\in\N$. So, by Fatou's Lemma,
\begin{equation}
\label{L2bo}
\sup_{t>0}\|u(t)\|_2\le 2\|a\|_2.
\end{equation}
(Notice that this argument does not require the use of the energy inequality. One could remove the
factor~$2$ by using the energy inequality. In the sequel, throughout the paper, we will avoid the use of the energy inequality  to make evidence that Theorem~\ref{thm;2} remains true for simpler toy models
that share the same scaling as Navier--Stokes.)

Let us go further with $L^2$-decay estimates. 
To start with,
 we begin by obtaining the quantified version of the  (non-optimal) decay 
$\|u(t)\|_2=O(t^{-3/8})$.
We have, by interpolation, $a\in L^{4n/(3+2n)}(\R^n)\subset L^1(\R^n)\cap L^2(\R^n)$. Moreover,
$\|u_0(t)\|_2\lesssim\|a\|_{4n/(3+2n)}t^{-3/8}$. For $k\in \N$ we have,
\begin{equation}
\label{decl2a}
\begin{split}
 \|u_{k+1}(t)\|_2
 &\le
\|u_0(t)\|_{2}
+c_1\int_0^{t}(t-s)^{-1/2}\|u_k(s)\|_2\|u_k(s)\|_\infty{\rm \,d}s\\
&\lesssim
\|u_0(t)\|_{2}
+c_1\int_0^{t}(t-s)^{-1/2}s^{-7/8}{\rm \,d}s\;
\|u_k\|_{X_\infty}\,\sup_{s>0}s^{3/8}\|u_k(s)\|_2
\\
&\lesssim
t^{-3/8}\Bigl(\|a\|_{4n/(3+2n)}+\|a\|_n\,\sup_{s>0}s^{3/8}\|u_k(s)\|_2\Bigr).
\end{split}
\end{equation}
With an appropriate choice of~$\delta>0$, (D3), by~\eqref{small-a}, iterating the above estimate
we obtain 
\begin{equation}
\label{hereD3}
\sup\limits_{k\in\N,t>0} t^{3/8}\|u_k(t)\|_2\lesssim \|a\|_{4n/(3+2n)}.
\end{equation}
 Hence,
\begin{equation}
 \label{est3/8}
 \begin{split}
 \sup_{t>0} \,t^{3/8}\|u(t)\|_2
 &\lesssim \|a\|_{4n/(3+2n)}.\\
 \end{split}
\end{equation}

Next we establish the quantified version of the improved, but still non-optimal, 
decay $\|u(t)\|_2=O(t^{-(n+1)/4})$.
Combining the standard estimate $\|e^{t\Delta}a\|_2\lesssim t^{-n/4}\|a\|_1$ with Lemma~\ref{lem:heat2},
we get the estimate
\begin{equation}
\label{hean1}
\sup_{t>0} \,
t^{(n+1)/4}\|e^{t\Delta}a\|_2
\lesssim  \|a\|_1^{1/2}\Bigl(\textstyle\int_{\R^n}|x|\,|a(x)|{\rm \,d}x\Bigr)^{1/2}.
\end{equation}
On the other hand, applying~\eqref{est:F}, \eqref{est3/8}, \eqref{hean1}, and that 
$\|u_k\|_{X_\infty}\lesssim \|a\|_n$ (see~\eqref{argu}), we get
\begin{equation}
\label{decl3}
\begin{split}
 \|u_{k+1}(t)\|_2
 &\le
\|u_0(t)\|_{2}
+c_2\int_0^{t/2}(t-s)^{-(n+2)/4}\|u_k(s)\|_2^2{\rm \,d}s
+c_1\int_{t/2}^{t}(t-s)^{-1/2}\|u_k(s)\|_2\|u_k(s)\|_\infty{\rm \,d}s\\
&\lesssim
\|u_0(t)\|_{2}
+ \|a\|_{4n/(3+2n)}^2\int_0^{t/2}(t-s)^{-(n+2)/4}s^{-3/4}{\rm \,d}s
+\|u_k\|_{X_\infty}\int_{t/2}^{t}(t-s)^{-1/2}s^{-1/2}\|u_k(s)\|_2{\rm \,d}s\\
&\lesssim
t^{-(n+1)/4}
\biggl(
\|a\|_1^{1/2}\Bigl(\textstyle\int_{\R^n}|x|\,|a(x)|{\rm \,d}x\Bigr)^{1/2}
+\|a\|_{4n/(3+2n)}^2
+\|a\|_n
\sup\limits_{s>0}s^{(n+1)/4}\|u_{k}(s)\|_2
\biggr).
\end{split}
\end{equation}
With an appropriate choice of~$\delta>0$, (D4), by~\eqref{small-a}, iterating the above estimate and taking $k\to\infty$ we obtain
\begin{equation}
\label{Jaa}
 \begin{split}
 \sup_{t>0}t^{(n+1)/4}\|u(t)\|_2
 &\lesssim
\|a\|_1^{1/2}\Bigl(\textstyle\int_{\R^n}|x|\,|a(x)|{\rm \,d}x\Bigr)^{1/2}
+\|a\|_{4n/(3+2n)}^2
=:J(a).
 \end{split}
\end{equation}

Of course we also have, because of~\eqref{L2bo}, 
\begin{equation}
 \label{small:tec}
 \begin{split}
 \|u(t)\|_2\lesssim \|a\|_2\wedge J(a)t^{-(n+1)/4},
 \end{split}
\end{equation}
where
the notation $\alpha \wedge\beta$ stands for $\min\{\alpha,\beta\}$.
The two terms defining $J(a)$ rescale in the same way: 
if $\lambda>0$ and $a_\lambda=\lambda\,a(\lambda\cdot)$, then we see that
$\lambda^{n-1/2}J(a_\lambda)=J(a)$. This makes~\eqref{small:tec} a scale invariant estimate.

Estimate~\eqref{small:tec} implies that $u\in L^2\bigl(\R^+;L^2(\R^n)\bigr)$.
More precisely, for any $\tau>0$,
\begin{equation*}
\begin{split}
\int_0^\infty \|u(s)\|_2^2\text{\,d}s 
&\lesssim
\int_0^{\infty} \bigl(\|a\|_2^2\wedge J(a)^2s^{-(n+1)/2} \bigr){\rm \,d}s\\
 &\le\int_0^{\tau}\|a\|_2^2{\rm \,d}s +\int_\tau^{\infty} J(a)^2s^{-(n+1)/2} {\rm \,d}s\\
 &\lesssim \tau\|a\|_2^2+J(a)^2\tau^{-(n-1)/2}.
\end{split}
\end{equation*}
Equalising the two last terms in the right-hand side leads to the choice $\tau=(J(a)/\|a\|_2)^{4/(n+1)}$.
Then we get the quantified scale-invariant $L^2\bigl(\R^+;L^2(\R^n)\bigr)$ estimate
\begin{equation}
\label{lateru}
\int_0^\infty \|u(s)\|_2^2\text{\,d}s 
 \lesssim J(a)^{4/(n+1)}\|a\|_2^{2(n-1)/(n+1)}.
\end{equation}

\paragraph{The optimal, scale invariant and quantified $L^2$-decay estimate.}
We are now in the position of obtaining the quantified version of the 
optimal decay $\|u(t)\|_2=O(t^{-(n+2)/4}$).
Indeed, applying Lemma~\ref{lem:heat2} we get
$\|u_0(t)\|_2\lesssim t^{-(n+2)/4}\int_{\R^n}|x|\,|a(x)|{\rm\,d}x$.
Moreover, by \eqref{small:tec}, \eqref{lateru}, that by construction hold true for the approximate
solutions $u_k$,
uniformly with respect to~$k\in\N$,
we obtain,
\begin{equation}
\label{decl2}
\begin{split}
 \|u_{k+1}(t)\|_2
 &\le
\|u_0(t)\|_{2}
+c_2\int_0^{t/2}(t-s)^{-(n+2)/4}\|u_k(s)\|_2^2{\rm \,d}s
+c_1\int_{t/2}^{t}(t-s)^{-1/2}\|u_k(s)\|_2\|u_k(s)\|_\infty{\rm \,d}s\\
&\lesssim
\|u_0(t)\|_{2}
+ t^{-(n+2)/4}\int_0^\infty \|u_k(s)\|_2^2\text{\,d}s
+\|u_k\|_{X_\infty}\int_{t/2}^{t}(t-s)^{-1/2}s^{-1/2}\|u_k(s)\|_2{\rm \,d}s\\
&\lesssim
t^{-(n+2)/4}\Bigl(\textstyle\int_{\R^n}|x|\,|a(x)|{\rm\,d}x
+ J(a)^{4/(n+1)}\|a\|_2^{2(n-1)/(n+1)}
+\|a\|_n\sup\limits_{s>0}s^{(n+2)/4}\|u_{k}(s)\|_2\Bigr).
\end{split}
\end{equation}
Once more, provided we make an appropriate choice of~$\delta>0$, from the smallness condition~\eqref{small-a} we obtain, by iteration,
\[
\sup_{t>0}t^{(n+2)/4}\|u(t)\|_2\lesssim
\textstyle\int_{\R^n}|x|\,|a(x)|{\rm\,d}x
+ J(a)^{4/(n+1)}\|a\|_2^{2(n-1)/(n+1)}=:K(a).
\]
This estimate, combined with~\eqref{L2bo} gives the quantified, scale-invariant, optimal estimate
\begin{equation}
 \label{opt-dec}
 \|u(t)\|_2\lesssim \|a\|_2\wedge  K(a)t^{-(n+2)/4},
\end{equation}
where
\begin{equation}
 \label{Ja}
\begin{split}
K(a)
&=
\textstyle\int_{\R^n}|x|\,|a(x)|{\rm\,d}x
+ J(a)^{4/(n+1)}\|a\|_2^{2(n-1)/(n+1)}\\
&\simeq
\textstyle\int_{\R^n}|x|\,|a(x)|{\rm\,d}x+
\|a\|_1^{2/(n+1)}\|a\|_2^{2(n-1)/(n+1)}\Bigl(\textstyle\int_{\R^n}|x|\,|a(x)|{\rm \,d}x\Bigr)^{2/(n+1)}
+\|a\|_{4n/(3+2n)}^{8/(n+1)}\|a\|_2^{2(n-1)/(n+1)}
\end{split}
\end{equation}
Observe that the functional $K$ satisfies the scaling property~\eqref{scaK}.
In 2D,  $K(a)$ can be rewritten as:
\[
K(a)\simeq
\textstyle\int_{\R^n}|x|\,|a(x)|{\rm\,d}x
+\|a\|_2^{2/3}\biggl(
\|a\|_1\bigl(\textstyle\int_{\R^n}|x|\,|a(x)|{\rm\,d}x\bigr)+\|a\|_{8/7}^4
\biggr)^{2/3}
\qquad(n=2).
\]

\subsection{The inductive procedure for the construction of  $\{f^{(m)}\}_{m\in\N}$ and
$\{u^{(m)}\}_{m\in\N}$.} 
\label{sec:induci}

We now generalise the procedure of the previous subsection, by first recursively generating a sequence of external forces $f^{(m)}$, next constructing the corresponding
 solutions $u^{(m)}$ of the Navier-Stokes equations:
\begin{equation*}\tag{N-S$_m$}
u^{(m)}(t)=e^{t\Delta}a 
+
\int_0^t e^{(t-s)\Delta}\PP\nabla\cdot f^{(m)}(s)\text{\,d}s
-
\int_0^t e^{(t-s)\Delta}\PP\nabla\cdot (u^{(m)}\otimes u^{(m)})(s)\text{\,d}s, 
\end{equation*}
for $m\in\N$ by Fujita--Kato's method.

First of all we set 
\[
f^{(0)}\equiv0,
\]
 so that $u^{(0)}$ is just the Navier-Stokes flow with vanishing external force constructed before.
Next, let $\phi$ be any compactly supported measurable function in $\R^n\times[0,\infty)$, such that
\begin{equation}
\label{inte1}
\phi\in L^\infty_c(\R^n\times \R^+)\quad
\text{and}
\quad\int_0^\infty\!\!\! \int_{\R^n} \phi(x,t)\text{\,d}x\text{\,d}t=1.
\end{equation}

Moreover,  we set, for
$m\in\N$,
\begin{equation}
\label{compon}
\begin{split}
c_{k\ell}^{(m)}
&=\int_0^\infty\!\!\!\int_{\R^n} 
u^{(m)}_k u_\ell^{(m)} \text{\,d}y\text{\,d}s,
\qquad k,\ell=1,\dots,n,
\\
\bar{c}^{(m)}
&=c_{11}^{(m)}+\dots+c_{nn}^{(m)},
\end{split}
\end{equation}
where $u^{(m)}=(u^{(m)}_1,\ldots, u^{(m)}_n)$,
and, for $m=1,2\ldots$,
\begin{equation}
\label{eq:recf}
f^{(m)}_{k\ell}(x,t)
=\begin{cases}
c_{k\ell}^{(m-1)}\phi(x,t) & k\neq \ell, \\
(c_{kk}^{(m-1)}-\bar{c}^{(m-1)})\phi(x,t) &k=\ell.
\end{cases}
\end{equation}
Moreover, we denote
\begin{equation}
 \label{def:im}
 I^{(m)}=\int_0^\infty\!\!\!\int_{\R^n} |u^{(m)}(x,s)|^2{\rm\,d}x{\rm\,d}s.
\end{equation}
At this stage, the definition of $f^{(m)}$,  $u^{(m)}$ and $I^{(m)}$, for $m\ge1$, is only formal.
For the solutions $u^{(m)}$ to be well-defined and global-in-time (it is a priori not obvious that the lifetime of the solution of (N-S$_m$) is indeed infinite),
we need to ensure that the external force $f^{(m)}$ does satisfy an appropriate smallness
condition.  We will specify this condition below, by prescribing an additional smallness condition on the function~$\phi$.
To make our definitions of $f^{(m)}$, $u^{(m)}$ and $I^{(m)}$  rigorous for all natural number~$m$, we will proceed by induction.

Notice first that our estimates in the previous subsection imply
\[
u^{(0)}\in X_{2n}
\quad\text{and}\quad u^{(0)}\in L^2\bigl(\R^+;L^2(\R^n)\bigr).
\]
The latter ensures that $I^{(0)}<\infty$, hence 
$c^{(0)}_{k\ell}$, $\bar c^{(0)}$, and so $f^{(1)}$, are all
well defined for $k,\ell\in\{1,\ldots,n\}$.
More precisely (see \eqref{lateru}):
\begin{equation}
\label{ioo}
 I^{(0)}\lesssim 
 J(a)^{4/(n+1)}\|a\|_2^{2(n-1)/(n+1)},
 \end{equation}
where $J(a)$ was defined in~\eqref{Jaa}.

Let $m\in\{1,2,\dots\}$.
Let us make the inductive assumptions (IA1)-(IA3) that a global mild solution 
$u^{(m-1)}$ of (N-S$_{m-1}$) does exist, 
such that
\begin{subequations}
\label{subab}
\begin{align}
\label{IA1}
\tag{IA1}
&u^{(m-1)}\in X_{2n},\\
\label{IA2}
\tag{IA2}
&u^{(m-1)}\in L^2\bigl(\R^+;L^2(\R^n)\bigr)
\end{align}
and that
\begin{equation}
\label{indui}
\tag{IA3}
I^{(m-1)}\le  L(a).
\end{equation}
\end{subequations}
Here, 
\begin{equation}
 \label{Lan}
L(a):=\gamma_n J(a)^{4/(n+1)}\|a\|_2^{2(n-1)/(n+1)},
\end{equation}
where the constant $\gamma_n>0$ (depending only on the space dimension) will be specified below.
The first condition to be imposed on $\gamma_n$ is the compatibility with~\eqref{ioo}, which is of course possible.

By \eqref{indui}, $I^{(m-1)}<\infty$ and so $f^{(m)}$ is well defined by~\eqref{eq:recf}.
Our next goal is to show that a global solution $u^{(m)}$ of (N-S$_m$) does exist, 
satisfying~\eqref{IA1}--\eqref{indui} with the integer $m$ instead of $m-1$.

We note that
\begin{equation*}
\begin{split}
\left\|\int_0^t 
 e^{(t-s)\Delta} \PP\nabla\cdot f^{(m)}(s)
\text{\,d}s\right\|_{2n}
&\lesssim\int_0^t (t-s)^{-1/2}
s^{-{3}/{4}}\text{\,d}s
\sup_{s>0}s^{3/4}\|f^{(m)}(s)\|_{2n}\\
&\lesssim t^{-1/4}
\sup_{s>0}s^{3/4}\|f^{(m)}(s)\|_{2n}.
\end{split}
\end{equation*}
Moreover, by the definition of~$f^{(m)}$, we see that 

\begin{equation*}
\|f^{(m)}(s)\|_p 
\lesssim I^{(m-1)} \|\phi(s)\|_p,
\end{equation*} 
for $1\leq p \le\infty$.
So we obtain that
\begin{equation*}
\begin{split}
\biggl\|
e^{t\Delta}a 
+ \int_0^t  e^{(t-s)\Delta} \PP\nabla\cdot f^{(m)}(s) \text{\,d}s
\biggr\|_{X_{2n}}
&\lesssim \|a\|_n+I^{(m-1)}\sup_{0<s<\infty}s^{3/4}\|\phi(s)\|_{2n}\\
&\lesssim \|a\|_n+L(a)\sup_{0<s<\infty}s^{3/4}\|\phi(s)\|_{2n}.
\end{split}
\end{equation*}
The first condition that we need to put on $\phi$ is
\begin{equation}
\label{small:phi1}
 L(a)\sup_{0<s<\infty}s^{3/4}\|\phi(s)\|_{2n} \le \|a\|_n.
 \tag{A1}
\end{equation}
Under condition~\eqref{small:phi1}, we deduce
\begin{equation*}
\biggl\|
e^{t\Delta}a 
+ \int_0^t  e^{(t-s)\Delta} \PP\nabla\cdot f^{(m)}(s) \text{\,d}s
\biggr\|_{X_{2n}} \lesssim \|a\|_n.
\end{equation*}
We can now deduce from the usual fixed point method and assumption~\eqref{small-a}
that, choosing $\delta>0$ small enough, (D5),
a solution $u^{(m)}$ of (N-S$_m$) does exist,
such that
\[
u^{(m)}\in X_{2n} \quad\text{and}\quad \|u^{(m)}\|_{X_{2n}}\lesssim \|a\|_n.
\]
Such solution $u^{(m)}$ can be obtained as the limit for $k\to\infty$, in the $X_{2n}$-norm, of the sequence of approximate solutions\footnote
{If necessary, we should put the vector symbol $\vec u^{(m)}_k$ just because we already used 
with a different meaning $u^{(m)}_k$ (=the $k$ component of the vector $u^{(m)}$, in~\eqref{compon}. 
For the simplicity, we omit the vector symbol since we can distinguish them from the context.} 
$u_k^{(m)}$ ($k=0,1,\ldots$)
\begin{equation}
\label{app-mk}
u^{(m)}_{k+1}(t)=e^{t\Delta}a+\int_0^t e^{(t-s)\Delta}\PP\nabla\cdot f^{(m)}(s){\rm\,d}s
+G( u^{(m)}_{k}, u^{(m)}_{k})(t),
\end{equation}
with
\[
 u^{(m)}_0(t)=e^{t\Delta}a+\int_0^t e^{(t-s)\Delta}\PP\nabla\cdot f^{(m)}(s){\rm\,d}s.
\]
Similarly as we did in subsection~\ref{sub:0}, we now verify that $u^{(m)}\in X_n\cap X_\infty$.
For this, we recall that (N-S$_{m}$) reads
$u^{(m)}(t)=e^{t\Delta}a+\int_0^t e^{(t-s)\Delta}\PP\nabla\cdot f(s){\rm\,d}s+G(u^{(m)},u^{(m)})$.
Recalling the arguments in~\eqref{argu} we only have to check that the forcing term
$ \int_0^t  e^{(t-s)\Delta} \PP\nabla\cdot f^{(m)}(s) \text{\,d}s$ does belong to $X_n\cap X_\infty$.
But this is immediate, as one can see applying twice~\eqref{est:F} with~$p=1$ and $p=(2n)'$, to get
\[
\begin{split}
\biggl\|
\int_0^t  e^{(t-s)\Delta} \PP\nabla\cdot f^{(m)}(s) \text{\,d}s
\biggr\|_{X_n} 
&\lesssim \sup_{s>0}s^{1/2}\|f^{(m)}(s)\|_n\\
& \lesssim L(a)\sup_{s>0}s^{1/2}\|\phi(s)\|_{n}.
\end{split}
\]
and
\[
\begin{split}
\biggl\|
\int_0^t  e^{(t-s)\Delta} \PP\nabla\cdot f^{(m)}(s) \text{\,d}s
\biggr\|_{X_\infty} 
&\lesssim \sup_{s>0}s^{3/4}\|f^{(m)}(s)\|_{2n}\\
&\lesssim  L(a) \sup_{s>0}s^{3/4}\|\phi(s)\|_{2n}.
\end{split}
\]
We now choose $\phi$ such that, in addition to~\eqref{small:phi1}, it satisfies
\begin{equation}
\label{small:phi2}
 L(a)\sup_{s>0}s^{1/2}\|\phi(s)\|_{n} \le \|a\|_n.
 \tag{A2}
\end{equation}
Then, just like we did in~\eqref{argu}, we get
\begin{align}
\|u^{(m)}\|_{X_n}&\lesssim \|a\|_n\\
\label{binf}
\|u^{(m)}\|_{X_{\infty}}&\lesssim\|a\|_n.
\end{align}

Let us go further with the relevant $L^2$-estimates for~$u$:
observe that, from~\eqref{est:F},
\begin{equation}
\label{qd9}
\begin{split}
 \biggl\|\int_0^t e^{(t-s)\Delta}\PP\nabla\cdot f^{(m)}(s){\rm\,d}s\biggr\|_2
 &\le \pi c_1\sup_{s>0} s^{1/2}\|f^{(m)}(s)\|_2\\
 &\le \pi c_1 I^{(m-1)}\sup_{s>0}s^{1/2}\|\phi(s)\|_2\\
 & \le \pi c_1 L(a)\sup_{s>0}s^{1/2}\|\phi(s)\|_2.
 \end{split}
\end{equation}
We now choose $\phi$ in a such way that  it satisfies also
\begin{equation}
\label{Ass3}
 \tag{A3}
 \pi c_1 L(a)\sup_{s>0}s^{1/2}\|\phi(s)\|_2
  \le \|a\|_2.
\end{equation}
Then, estimating the approximate solutions $u^{(m)}_k$ for $k\in\N$, in the same way as we argued to
obtain~\eqref{L2bo}, from~\eqref{app-mk}, we obtain
\[
\begin{split}
 \|u^{(m)}_{k+1}(t)\|_2
 &\le
2\|a\|_{2}
+c_1\int_0^{t}(t-s)^{-1/2}\|u^{(m)}_k(s)\|_2\|u^{(m)}_k(s)\|_\infty{\rm \,d}s\\
&\le 
2\|a\|_2
+\pi c_1\|u_k^{(m)}\|_{X_\infty}\sup_{s>0}\|u^{(m)}_k(s)\|_2.\\
\end{split}
\]
As $\pi c_1\|u^{(m)}_k\|_{X_\infty}\lesssim\|a\|_n$, from~\eqref{small-a},  by an appropriate choice of $\delta>0$, (D6),
we can ensure that
\[
\pi c_1\|u^{(m)}_k\|_{X_\infty}\le1/2<1
\]
and by iteration
\begin{equation}
\label{uborn}
\begin{split}
\sup_{t>0}\|u^{(m)}(t)\|_2
&\le 
4\|a\|_{2}.
\end{split}
\end{equation}

Next step consists in proving decay estimates for $\|u^{(m)}(t)\|_2$.
As in the previous section, we begin establishing a (uniform-in-$m$) quantified version of the decay estimate
$\|u^{(m)}(t)\|_2=O(t^{-3/8})$.
Let us recall that $\|e^{t\Delta}a\|_2\lesssim  \|a\|_{4n/(3+2n)} t^{-3/8}$.
Concerning the forcing term, we have
 \begin{equation}
 \label{f38}
\begin{split}
 \biggl\|\int_0^t e^{(t-s)\Delta}\PP\nabla\cdot f^{(m)}(s){\rm\,d}s\biggr\|_ 2
 &\lesssim
 \Bigl(\int_0^t(t-s)^{-1/2}s^{-7/8}{\rm\,d}s\Bigr)\sup_{s>0}s^{7/8}\|f^{(m)}(s)\|_2\\
 &\lesssim t^{-3/8} L(a)\sup_{s>0}s^{7/8}\|\phi(s)\|_2.
 \end{split}
\end{equation}
This leads us to add another condition to $\phi$, namely
\begin{equation}
\label{Ass4}
 \tag{A4}
 L(a)\sup_{s>0}s^{7/8}\|\phi(s)\|_2 \le \|a\|_{4n/(3+2n)}.
 \end{equation}

Therefore
\begin{equation}
\sup_{t>0} t^{3/8}\biggl\|e^{t\Delta}a
  +\int_0^t e^{(t-s)\Delta}\PP\nabla\cdot f^{(m)}(s){\rm\,d}s     \biggr\|_2 
  \lesssim  \|a\|_{4n/(3+2n)}.
\end{equation}

We can now reproduce the same calculations as in~\eqref{decl2a}-\eqref{est3/8} 
with an appropriate choice of $\delta>0$, (D7), by (S)
and we obtain
\begin{equation}
 \label{est-m3/8}
 \sup_{t>0} \,t^{3/8}\|u^{(m)}(t)\|_2
\lesssim \|a\|_{4n/(3+2n)}.
\end{equation}

Next, let us prescribe the conditions on $\phi$ in order to obtain 
a (uniform-in-$m$) quantified version of the estimate $\|u^{(m)}(t)\|_2=O(t^{-(n+1)/4})$.
First of all, let us recall estimate~\eqref{hean1}, that reads
\begin{equation*}
 \label{lin1}
 \|e^{t\Delta}a\|_2\lesssim \Bigl(\int_{\R^n}|x|\, |a(x)|{\rm\,d}x\Bigr)^{1/2}\|a\|_1^{1/2}t^{-(n+1)/4}.
\end{equation*}
To get the same decay rate for the forcing term, we set $\frac1q=1-\frac{1}{2n}$. Using~\eqref{est:F} with $p=2n/(n+1)$,
we obtain 
\begin{equation*}
\begin{split}
 \biggl\|\int_0^t e^{(t-s)\Delta}\PP\nabla\cdot f^{(m)}(s){\rm\,d}s\biggr\|_ 2
&\lesssim  t^{-(n+1)/4}\int_0^{t/2}\|f^{(m)}(s)\|_q{\rm\,d}s
  +\int_{t/2}^t(t-s)^{-1/2}\|f^{(m)}(s)\|_2{\rm\,d}s \\
&\lesssim  t^{-(n+1)/4} L(a)\Bigl(\int_0^\infty\|\phi(s)\|_q{\rm\,d}s
  +  \sup_{s>0}s^{(n+3)/4}\|\phi(s)\|_2\Bigr).
  \end{split}
 \end{equation*}
We then choose $\phi$ in such a way that
\begin{equation}
 \label{Ass5}
  L(a)\Bigl(\int_0^\infty\|\phi(s)\|_{2n/(2n-1)}{\rm\,d}s
  +  \sup_{s>0}s^{(n+3)/4}\|\phi(s)\|_2\Bigr)
 \le 
  \Bigl(\int_{\R^n}|x|\, |a(x)|{\rm\,d}x\Bigr)^{1/2}\|a\|_1^{1/2}.
  \tag{A5}
\end{equation}
Then we get
\begin{equation}
 \sup_{t>0} t^{(n+1)/4}\biggl\|e^{t\Delta}a+\int_0^t e^{(t-s)\Delta}\PP\nabla\cdot f^{(m)}(s){\rm\,d}s\biggr\|_ 2
 \lesssim  \Bigl(\int_{\R^n}|x|\, |a(x)|{\rm\,d}x\Bigr)^{1/2}\|a\|_1^{1/2}.
\end{equation}
Then, reproducing the same computations as in~\eqref{decl3}-\eqref{Jaa}
with an appropriate choice of $\delta>0$, (D8), by
(S),
we get
\begin{equation}
\label{jia}
 \begin{split}
 \sup_{t>0}t^{(n+1)/4}\|u^{(m)}(t)\|_2
 &\lesssim
\|a\|_1^{1/2}\Bigl(\textstyle\int_{\R^n}|x|\,|a(x)|{\rm \,d}x\Bigr)^{1/2}
+\|a\|_{4n/(3+2n)}^2\\
&=J(a).
 \end{split}
\end{equation}
>From the latter estimate, and the already established uniform bound $\|u^{(m)}(t)\|_2\lesssim \|a\|_2$,
arguing as we did before in~\eqref{lateru}, we get
\begin{equation}
\label{laterum}
\int_0^\infty \|u^{(m)}(s)\|_2^2\text{\,d}s 
 \lesssim J(a)^{4/(n+1)}\|a\|_2^{2(n-1)/(n+1)}.
\end{equation} 
Now, let us be more explicit with the constants that appear in two of the estimates obtained before:
From~\eqref{ioo}, we infer that there exists $K_n>0$ (independent $a$) such that
\begin{equation}
\label{ioo-k}
 I^{(0)}\le K_n
 J(a)^{4/(n+1)}\|a\|_2^{2(n-1)/(n+1)}.
 \end{equation}
From~\eqref{laterum}, we  deduce the existence of $K'_n>0$ (independent on $m$ and $a$) such that
 \begin{equation}
\label{laterum-k}
I^{(m)}
 \le K'_n J(a)^{4/(n+1)}\|a\|_2^{2(n-1)/(n+1)} \qquad(m=1,2,\ldots).
\end{equation}
This leads us to choose, in \eqref{Lan}, $\gamma_n=\max\{K_n,K'_n\}$.
It then follows that $I^{(0)}\le L(a)$ and that
$I^{(m)}\le L(a)$. This allows to close the inductive argument~\eqref{indui}.

Summarising, we proved the existence of $\delta>0$ (depending only on the space dimension, in agreement with conditions (D1)--(D8)) such that, 
under the smallness assumption $\|a\|_n\le \delta$ and the smallness conditions
(A1)--(A5) on $\phi$,  there exists a sequence of global solutions $u^{(m)}$ $(m=0,1,2,\ldots)$
of (N-S$_m$), which is bounded in $X_{2n}$, and also bounded in $X_n\cap X_\infty$, in 
$L^\infty\bigl(\R^+;L^2(\R^n)\bigr)$,
and in $L^2\bigl(\R^+;L^2(\R^n)\bigr)$.
Strengthening a little bit the smallness conditions on $\phi$, it would be possible in the same way to 
get also the boundedness of $u^{(m)}$ under the stronger norm
$v\mapsto\sup\limits_{t>0}(1+t)^{(n+2)/4}\|v(t)\|_2$.

\subsection{Convergence of $\{u^{(m)}\}_{m\in\N}$.}
In this section we study the convergence of 
$u^{(m)}$ in the space
\begin{equation}
 \label{espY}
 Y=\{v\in L^\infty\bigl(\R^+;L^2(\R^n)\bigr)\colon \|v\|_Y:=\sup_{s>0}(1+s)^{(n+1)/4}\|v(s)\|_2<\infty\}.
\end{equation}
This is not a scale invariant space: what it does matter here is that~$Y$ is imbedded in $L^2\bigl(\R^+;L^2(\R^n)\bigr)$.
This non-invariance explains why an additional artificial smallness assumption on~$a$
appears in the next Lemma. In any
case, this artificial assumption
will be removed at the end of the proof of our theorem.

\begin{lemma}
\label{lem-sa}
Let $u^{(m)}$ be the  of solution of (N-S$_m$) constructed in the previous section.
There exist two constants 
$\delta,\delta^\prime>0$, depending only on the space dimension, such that
if $\|a\|_n\le \delta$ and 
\begin{equation}
\label{small-st}
\|a\|_n^{1/n} (J(a)^{1-1/n}+\|a\|_2^{1-1/n})\le \delta^\prime,
\tag{S$^\prime$}
\end{equation}
where $J(a)$ was defined in~\eqref{Jaa}, then,  for $m\in\N$,
\begin{equation}
 \label{est:Y}
 \|u^{(m+1)}-u^{(m)}\|_Y
 \lesssim
  \Bigl\| \int_0^t e^{(t-s)\Delta}\PP\nabla\cdot
[f^{(m+1)}-f^{(m)}](s)\text{\,d}s \Bigr\|_{Y}.
\end{equation}
\end{lemma}

\begin{proof}
We note that
\begin{equation*}
\begin{split}
u^{(m+1)}(t)-u^{(m)}(t) &=
\int_0^t e^{(t-s)\Delta}\PP\nabla\cdot
[f^{(m+1)}-f^{(m)}](s)\text{\,d}s
\\
&\quad +\int_0^t 
e^{(t-s)\Delta}\PP\nabla\cdot
\bigl[(u^{(m+1)}-u^{(m)})\otimes u^{(m+1)}\bigr](s)\text{\,d}s 
\\
&\quad +
\int_0^t e^{(t-s)\Delta}\PP\nabla\cdot
\bigl[u^{(m)}\otimes(u^{(m+1)}-u^{(m)})\bigr](s)\text{\,d}s
\\
&=:\mathcal{I}_1^{(m)}(t) 
+ \mathcal{I}_2^{(m)}(t)
+\mathcal{I}_3^{(m)}(t).
\end{split}
\end{equation*}
We make two separate estimates for~$\mathcal{I}_2^{(m)}$.
The first one  will be useful for $0\le t\le 1$ and relies on~\eqref{binf}:
\begin{equation*}
\begin{split}
\|\mathcal{I}_2(t)\|_2
&\le
c_{1}
\int_0^{t} (t-s)^{-1/2} \|u^{(m+1)}(s)-u^{(m)}(s)\|_2\,\|u^{(m+1)}(s)\|_{\infty}\text{\,d}s\\
&\lesssim
\|u^{(m+1)}-u^{(m)}\|_Y\|u^{(m+1)}\|_{X_\infty}\int_0^t(t-s)^{-1/2}(1+s)^{-(n+1)/4}s^{-1/2}\text{\,d}s\\
&\lesssim
\|u^{(m+1)}-u^{(m)}\|_Y\|a\|_n.
\end{split}
\end{equation*}

The second estimate of~$\mathcal{I}_2^{(m)}(t)$ will be useful for $t\ge1$. It is obtained using~\eqref{est:F} in the first inequality, interpolation in the second,~\eqref{binf}, \eqref{uborn} and~\eqref{jia} in the third inequality below:
\begin{equation*}
\begin{split}
\|\mathcal{I}_2(t)\|_2
&\le
c_{2n/(n+1)}
\int_0^{t/2} (t-s)^{-(n+1)/4} \|u^{(m+1)}(s)-u^{(m)}(s)\|_2\,\|u^{(m+1)}(s)\|_{2n/(n-1)}\text{\,d}s\\
&\qquad\qquad
  +c_{1}
\int_{t/2}^t(t-s)^{-1/2} \|u^{(m+1)}(s)-u^{(m)}(s)\|_2\,\|u^{(m+1)}(s)\|_\infty\text{\,d}s\\
&\lesssim
t^{-(n+1)/4} \|u^{(m+1)}-u^{(m)}\|_Y
\int_0^{t/2}(1+s)^{-(n+1)/4}\|u^{(m+1)}(s)\|_2^{1-1/n}\|u^{(m+1)}(s)\|_\infty^{1/n}\text{\,d}s\\
&\qquad\qquad
  + (1+t)^{-(n+1)/4}\|u^{(m+1)}-u^{(m)}\|_Y\|u^{(m)}\|_{X_\infty}\\
&\lesssim
 t^{-(n+1)/4}\|a\|_n^{1/n} \|u^{(m+1)}-u^{(m)}\|_Y
  \Bigl(
   \int_0^{t/2}(1+s)^{-(n+1)/4} \bigl(J(a)s^{-(n+1)/4}\wedge\|a\|_2\bigr)^{(1-1/n)}s^{-1/(2n)}\text{\,d}s
      + \|a\|_n^{1-1/n}\Bigr)\\
&\lesssim      
 t^{-(n+1)/4}\|a\|_n^{1/n} \|u^{(m+1)}-u^{(m)}\|_Y
  \Bigl(
  \ J(a)^{1-1/n}+\|a\|_2^{1-1/n}+\|a\|_n^{1-1/n}\Bigr).
 \end{split}
\end{equation*}

Combining the two previous  estimates for $\mathcal{I}_2^{(m)}(t)$, we deduce
\begin{equation}
\label{est:i2}
\|\mathcal{I}_2^{(m)}\|_Y
\lesssim \|a\|_n^{1/n}  \Bigl(
  \ J(a)^{1-1/n}+\|a\|_2^{1-1/n}+\|a\|_n^{1-1/n}\Bigr) \|u^{(m+1)}-u^{(m)}\|_Y.
\end{equation}
In the same manner, we have
\begin{equation}
\label{est:i3}
\|\mathcal{I}_3^{(m)}\|_Y
\lesssim \|a\|_n^{1/n}  \Bigl(
  \ J(a)^{1-1/n}+\|a\|_2^{1-1/n}+\|a\|_n^{1-1/n}\Bigr) \|u^{(m+1)}-u^{(m)}\|_Y.
\end{equation}
Now, if $\delta>0$,
is small enough, (D9), and $\delta^\prime>0$ is suitably small as well,
then we deduce~\eqref{est:Y} from~\eqref{est:i2}-\eqref{est:i3}.
\end{proof}

Throughout the remaining part of this subsection, we assume that $a$ fulfils the smallness condition~\eqref{small-st}, along with~\eqref{small-a}. For the convergence of $\{u^{(m)}\}$, we choose $\delta>0$ in a such way
that conditions (D1)-(D9) are fulfilled. This ensures that all our previous estimates hold true.

Next, we estimate $\mathcal{I}_1^{(m)}(t)$: 
for $0<t\leq 1$, we have
\begin{equation*}
\begin{split}
\|\mathcal{I}_1^{(m)}(t)\|_2
&\leq
c_1\int_0^t  (t-s)^{-1/2} \|f^{(m+1)}(s)-f^{(m)}(s)\|_2\text{\,d}s
\\
&\lesssim
 \sup_{s>0}\, s^{1/2}
\|f^{(m+1)}(s)-f^{(m)}(s)\|_2.
\end{split}
\end{equation*}
But
\begin{equation*}
\begin{split}
\|f^{(m+1)}(s)-f^{(m)}(s)\|_2
&\leq
\sum_{k,\ell=1}^n 
\|f_{k\ell}^{(m+1)}(s)-f_{k\ell}^{(m)}(s)\|_2
\\
&=\sum_{k\neq \ell}
\bigl|c_{k\ell}^{(m)}-c_{k\ell}^{(m-1)}
\bigr| \|\phi(s)\|_2
+
\sum_{k=1}^n
\bigl|(c_{kk}^{(m)}-\bar{c}^{(m)})-
(c_{kk}^{(m-1)}-\bar{c}^{(m-1)})\bigr|
\|\phi(s)\|_2
\end{split}
\end{equation*}
For the case $k\neq \ell$, applying~\eqref{est-m3/8} in the last inequality, and the fact that $\int_0^\infty(1+s)^{-(n+1)/4}s^{-3/8}\text{\,d}s<\infty$,  we have:
\begin{equation*}
\begin{split}
\bigl|c_{k\ell}^{(m)}-c_{k\ell}^{(m-1)}
\bigr| 
&=\left|
\int_0^\infty\!\!\!\int_{\R^n}u_k^{(m)}u_\ell^{(m)}
-u_k^{(m-1)}u_\ell^{(m-1)}\text{\,d}y\text{\,d}s
\right|
\\
&\lesssim
\int_0^\infty\|u^{(m)}(s)-u^{(m-1)}(s)\|_2 \bigl(\|u^{(m)}(s)\|_2+\|u^{(m-1)}\|_2\bigr)\text{\,d}s\\
 &\lesssim
 \|u^{(m)}-u^{(m-1)}\|_Y \|a\|_{4n/(3+2n)}.
\end{split}
\end{equation*}

For the case $k=\ell$ we  have
\begin{equation*}
\begin{split}
\bigl|(c_{kk}^{(m)}-\bar{c}^{(m)})-
(c_{kk}^{(m-1)}-\bar{c}^{(m-1)})\bigr|
&=\sum_{i\not=k} \bigl|(c_{ii}^{(m)}-c_{ii}^{(m-1)})\bigr|,\\
\end{split}
\end{equation*}
so that the same estimates as before applies.
Therefore,
\begin{equation*}
\|f^{(m+1)}(s)-f^{(m)}(s)\|_2
\lesssim
\|\phi(s)\|_2 \|a\|_{4n/(3+2n)}\|u^{(m)}-u^{(m-1)}\|_Y .
\end{equation*}

Therefore we obtain that for $0<t<1$
\[
\|\mathcal{I}_1^{(m)}(t)\|_2
\lesssim \Bigl(\sup_{s<1}s^{1/2}\|\phi(s)\|_2\Bigr)\|a\|_{4n/(3+2n)}\|u^{(m)}-u^{(m-1)}\|_Y  .
\]

For $t>1$, by similar way, we have with $\frac{1}{q}=1-\frac{1}{2n}$ 
\begin{equation*}
\begin{split}
\|\mathcal{I}_1^{(m)}(t)\|_2
&\lesssim 
\int_0^{t/2} (t-s)^{-\frac{n+1}{4}} 
\|f^{(m+1)}(s)-f^{(m)}(s)\|_q 
{\rm \,d}s
+
\int_{t/2}^t
(t-s)^{-1/2}
\|f^{(m+1)}(s)-f^{(m)}(s)\|_2{\rm \, d}s
\\
&\lesssim
t^{-\frac{n+1}{4}}
\Bigl(
\int_0^\infty \|\phi(s)\|_q {\rm \,d}s
+
\sup_{s>0} s^{\frac{n+3}{4}}\|\phi(s)\|_2
\Bigr)\|a\|_{4n/(3+2n)}\|u^{(m)}-u^{(m-1)}\|_Y. 
\end{split}
\end{equation*}
Then we obtain 
\begin{equation*}
\|\mathcal{I}_1^{(m)}\|_Y
\lesssim
\Bigl(
\sup_{s<1}s^{1/2}\|\phi(s)\|_2+
\int_0^\infty \|\phi(s)\|_{2n/(2n-1)} {\rm \,d}s
+
\sup_{s>0} s^{\frac{n+3}{4}}\|\phi(s)\|_2
\Bigr)
\|a\|_{4n/(3+2n)}\|u^{(m)}-u^{(m-1)}\|_Y.
\end{equation*}

By Lemma~\ref{lem-sa},
 there exists a constant $\delta''>0$, only depending on the space dimension, such that if
\begin{equation}
 \label{Ass6}
  \Bigl(\sup_{s<1}s^{1/2}\|\phi(s)\|_2
  +
\int_0^\infty \|\phi(s)\|_{2n/(2n-1)} {\rm \,d}s
+
\sup_{s>0} s^{\frac{n+3}{4}}\|\phi(s)\|_2\Bigr) \|a\|_{4n/(3+2n)}\le \delta'',
 \tag{A6}
\end{equation}
then
\[
\|u^{(m+1)}-u^{(m)}\|_Y\le \textstyle \frac12 \|u^{(m)}-u^{(m-1)}\|_Y.
\]
This implies that $u^{(m)}$ converges in the $Y$-norm to some limit $v\in Y$.

The convergence $u^{(m)}\to v$ in $Y$ implies that $u^{(m)}\to v$ in  
$L^2\bigl(\R^+;L^2(\R^n)\bigr)$.
The fact that $v\in L^2\bigl(\R^+;L^2(\R^n)\bigr)$ allows us to define
\begin{equation*}
\begin{split}
f^{(\infty)}_{k\ell}(x,t)
&=\begin{cases}
c_{k\ell}^{(\infty)}\phi(x,t) & k\neq \ell, \\
(c_{kk}^{(\infty)}-\bar{c}^{(\infty)})\phi(x,t) &k=\ell,
\end{cases}
\\
c_{k\ell}^{(\infty)}
&=\int_0^\infty\!\!\!\int_{\R^n} 
v_k v_\ell \text{\,d}y\text{\,d}s,
\qquad k,\ell=1,\dots,n,
\\
\bar{c}^{(\infty)}
&=c_{11}^{(\infty)}+\dots+c_{nn}^{(\infty)}.
\end{split}
\end{equation*}

In particular, we see that
$c_{k\ell}^{(m)} \to c_{k\ell}^{(\infty)}$ and
$\bar{c}^{(m)}\to \bar{c}^{(\infty)}$ as $m\to \infty$. 
Moreover, it holds that $\|f_{k\ell}^{(m)}(t)- f_{k\ell}^{(\infty)}(t)\|_p\to 0$ as $m\to \infty$ for all $t \geq 0$ and for all $1\leq p \leq \infty$. 

Hence,
$\int_0^t 
 e^{(t-s)\Delta} \PP \nabla\cdot f^{(m)}(s)\text{\,d}s
\to
\int_0^t 
 e^{(t-s)\Delta} \PP\nabla\cdot f^{(\infty)}(s)\text{\,d}s$
 for all $t>0$.
The strong convergence of $u^{(m)}$ to $v$ in $Y$ allows to
pass to the limit in the nonlinear term of (N-S$_m$).
This implies that the limit $v$ satisfies the Navier--Stokes equations, with force 
$\nabla\cdot f^{(\infty)}$, and initial data $a$, in its integral form:
\begin{equation*}
v(t)=e^{t\Delta}a
+
\int_0^t 
e^{(t-s)\Delta} \PP\nabla\cdot f^{(\infty)}(s)\text{\,d}s
-
\int_0^t
  e^{(t-s)\Delta}\PP \nabla\cdot
(v\otimes v)(s)\text{\,d}s,
\end{equation*}
for all $t>0$.

Now that we know that $v$ and $f^{(\infty)}$ are well defined, one can easily prove, in a similar way, that $u^{(m)}\to v$ in $X_{2n}$, then,  in other spaces, e.g., in $BC\bigl(\R^+;L^n(\R^n)\bigl)$ by \eqref{small-a}, \eqref{small-st} and (A1)--(A6).
The convergence $u^{(m)}\to v$ could be proved to be true also  in $L^2$ with the optimal time-weight, 
i.e., $\sup\limits_{t>0}(1+t)^{(n+2)/2}\|u^{(m)}(t)-v(t)\|_2\to0$, if we suitable strengthen the smallness conditions, but this will not needed  to prove our theorem.

\subsection{Asymptotic  behavior of $v(t)$ as $t\to \infty$}

As before, we work in this section under
 the smallness assumptions on~$a$~\eqref{small-a}-\eqref{small-st} and all the smallness conditions
 on~$\phi$ previously specified.

In this subsection, we discuss the large time decay rate of the Lebesgue norms of $v(t)$, by applying the asymptotic profiles of 
Miyakawa and Schonbek
\cite{Miyakawa Schonbek}.
We put
\begin{equation*}
\beta_{k,\ell}=\int_0^\infty\!\!\!\int_{\R^n}v_k v_\ell\text{\,d}y\text{\,d}s-
\int_0^\infty\!\!\! \int_{\R^n} f^{(\infty)}_{k\ell}\text{\,d}y\text{\,d}s
\end{equation*}

Here, when $k\neq \ell$, we see that
\begin{equation*}
\beta_{k,\ell}=\int_0^\infty\!\!\!\int_{\R^n}
v_kv_{\ell}\text{\,d}y\text{\,d}s
-c_{k\ell}^{(\infty)}\int_0^\infty\!\!\! \int_{\R^n}\phi(y,s)\text{\,d}y\text{\,d}s=0.
\end{equation*}
When $k=\ell$, we see that
\begin{equation*}
\beta_{k,k}= 
\int_0^\infty\!\!\!\int_{\R^n}
v_k^2\text{\,d}y\text{\,d}s-(c_{kk}^{(\infty)}-\bar{c}^{(\infty)})
\int_0^\infty\!\!\! \int_{\R^n}\phi(y,s)\text{\,d}y\text{\,d}s=
\bar{c}^{(\infty)}.
\end{equation*}
Therefore,
\begin{equation*}
\beta_{k,\ell}=c\delta_{k\ell}
\quad\text{for some }c\in\R, \quad\text{and $k,\ell=1,\ldots,n$}.
\end{equation*}
Now, the application of Corollary~\ref{cor;MS} gives the desired conclusions for $v$, namely:
\begin{equation*}
\label{fast-decay-av}
\lim_{t\to\infty}
t^{\frac{1}{2}+\frac{n}{2}(1-\frac{1}{q})}
\left\|
v(t)-e^{t\Delta}a
\right\|_q=0, \qquad 1\le q\le\infty
\end{equation*}
and, under the additional
 moment condition $\int_{\R^n} y_ka_j(y)\text{\,d} y=0$ for all $j,k=1,\ldots,n$,
\begin{equation*}
\label{fast-decay-v}
\lim_{t\to\infty}
t^{\frac{1}{2}+\frac{n}{2}(1-\frac{1}{q})}
\left\|
v(t)
\right\|_q=0, \qquad 1\le q\le\infty.
\end{equation*}

\subsection{Completion of proof of Theorem \ref{thm;2}}

Let us first complete the proof under the additional artificial smallness assumption~\eqref{small-st}
on~$a$.
To finish the proof in this case, we only have to collect 
all the previous needed  conditions on~$\phi$ (namely,~\eqref{small:phi1}, \eqref{small:phi2},
\eqref{Ass3}, \eqref{Ass4}, \eqref{Ass5} and~\eqref{Ass6}) 
and to ensure their compatibility.

Let us recall that $J(a)$ and $L(a)$ were defined by~\eqref{Jaa} and~\eqref{Lan}:
\begin{equation*}
 J(a)=\|a\|_1^{1/2}\Bigl(\textstyle\int_{\R^n}|x|\,|a(x)|{\rm \,d}x\Bigr)^{1/2}
+\|a\|_{4n/(3+2n)}^2
\end{equation*}
and
\begin{equation*}
L(a)=\gamma_n J(a)^{4/(n+1)}\|a\|_2^{2(n-1)/(n+1)},
\end{equation*}
where $\gamma_n$ only depends on $n$.
So, the conditions to be imposed on $\phi$ are:
\[
\phi\in L^\infty_c(\R^n\times \R^+),\qquad \int_0^\infty\!\!\!\int_{\R^n}\phi(x,t)\text{\,d}x\text{\,d}t=1,
\]
and
\begin{equation}
\label{all-small}
\begin{split}
& L(a)\sup_{0<s<\infty}s^{3/4}\|\phi(s)\|_{2n} \le \|a\|_n,\\
& L(a)\sup_{0<s<\infty}s^{1/2}\|\phi(s)\|_{n} \le \|a\|_n,\\
& c_1L(a)\sup_{s>0}s^{1/2}\|\phi(s)\|_2\le \|a\|_2,\\
& L(a)\sup_{s>0}s^{7/8}\|\phi(s)\|_2 \le \|a\|_{4n/(3+2n)}\\
&  L(a)\Bigl(\textstyle\int_0^\infty\|\phi(s)\|_{2n/(n-1)}{\rm\,d}s
  +  \sup_{s>0}s^{(n+3)/4}\|\phi(s)\|_2\Bigr)
 \le   \Bigl(\int_{\R^n}|x|\, |a(x)|{\rm\,d}x\Bigr)^{1/2}\|a\|_1^{1/2},\\
&  \Bigl(\sup_{s<1}s^{1/2}\|\phi(s)\|_2
  +
\int_0^\infty \|\phi(s)\|_{2n/(2n-1)} {\rm \,d}s
+
\sup_{s>0} s^{\frac{n+3}{4}}\|\phi(s)\|_2\Bigr) \|a\|_{4n/(3+2n)}\le \delta'',\\
\end{split}
\tag{A\,1--6}
\end{equation}
Recall that $\delta''>0$ is some constant depending  only on $n$.
In fact, nowhere we really needed that $\phi\in L^\infty_c(\R^n\times\R^+)$: the more general
condition $\phi\in L^1(\R^n\times\R^+)$, together with~\eqref{all-small} would have been enough.
In any case, the most obvious way to construct such a function $\phi$
is to start from a compacly supported and essentially bounded function 
$\Phi\in L^\infty_c(\R^n\times \R^+)$, supported in a cube 
$K=[-M,M]^n\times[0,T_0]$ for some $M>0$ and $T_0>0$, such that  
$\int_0^\infty\!\!\!\int_{\R^n}\Phi(x,t)\text{\,d}x\text{\,d}t=1$.
But,  for any $1<p\le\infty$, we have $\|R^{-n}\Phi(\cdot/R,t)\|_p\to0$ as $R\to+\infty$. Therefore, there exists
\[
R_0=R_0\bigl(n,M,T_0,\|a\|_1,\|a\|_n,\textstyle\int_{\R^n}|x|\,|a(x)|\text{\,d}x\bigr),
\]
(with $R_0$ depending only on the above mentioned parameters), such that, if we take $R\ge R_0$ and
\[
\phi(x,t)=R^{-n}\Phi(x/R,t),
\]
then the function $\phi$ satisfies all the required conditions~\eqref{all-small}.

This completes the proof of Theorem~\ref{thm;2}, at least under the additional smallness condition~\eqref{small-st} for~$a$, that was needed in Lemma~\ref{lem-sa}. So, it now remains
to show that this smallness condition~\eqref{small-st} can be removed.

This can be done via a standard scaling argument. Indeed, assuming (according with the assumptions
of Theorem~\ref{thm;2}) that 
 $\int |a(x)|(1+|x|)\text{\,d}x<\infty$, and that  $\|a\|_n<\delta$, we consider the rescaled data $a_\lambda=\lambda a(\lambda\cdot)$.
 We have $\|a_\lambda\|_n=\|a\|_n$, $\|a_\lambda\|_2=\lambda^{1-n/2}\|a\|_2$ and  $J(a_\lambda)=\lambda^{-(n-1/2)}J(a)$.
 
 So condition~\eqref{small-st} is fulfilled for the rescaled initial data $a_{\lambda}$, provided we choose $\lambda\ge\lambda_0$ for some $\lambda_0=\lambda_0(n)$ depending only on the space-dimension. This is true also when $n=2$, as in this case $\|a\|_2$ is itself assumed to be small.
By what we proved so far, we can construct an external force $\nabla\cdot f_\lambda$ as above (with $f_\lambda$ with compact support in space-time) and a solution $u_{\lambda}$ of the Navier--Stokes equation arising from $a_\lambda$ with force $\nabla\cdot f_{\lambda}$, that satisfies 
$\|u_\lambda(t)-e^{t\Delta}a_\lambda\|_q=o(t^{-\frac12-\frac n2(1-\frac1q)})$ as $t\to\infty$ for all $1\le q\le\infty$.
Now, let us set $u(x,t)=\lambda^{-1} u_\lambda(\lambda^{-1} x,\lambda^{-2} t)$ and 
$f(x,t)=\lambda^{-2} f_\lambda (\lambda^{-1}x,\lambda^{-2}t)$.
Then $f$ is compactly supported in space-time and $u$ solves (N-S). Moreover, 
$\|u(t)-e^{t\Delta}a\|_q=o(t^{-\frac12-\frac n2(1-\frac1q)})$ as $t\to\infty$.

Moreover, $a$ has vanishing first-order moments if and only if the same is true for $a_\lambda$.
Therefore, we get under this moment condition the rapid dissipation property 
$\|u(t)\|_q=o(t^{-\frac12-\frac n2(1-\frac1q)})$ as $t\to\infty$.

This  fully establishes Theorem \ref{thm;2}.

\medskip
\textbf{Acknowledgement.} The authors are grateful to Professor Taku Yanagisawa for his comments and to the referees for their careful reading.
The second author is partially supported by JSPS Grant-in-Aid for Young Scientists (B) 17K14215.


\end{document}